\documentclass[12pt]{article}
\usepackage{color,amsfonts, amsmath, amssymb,latexsym,amsthm}
\usepackage[margin=1.1in]{geometry}
\usepackage{tikz-cd}

\def\red#1{\textcolor[rgb]{1,0.00,0.00}{#1}}
\def\???{\red{???}}

\usepackage{marginnote}

\newtheorem{theorem}{Theorem}[section]
\newtheorem{lemma}[theorem]{Lemma}

\newtheorem{proposition}[theorem]{Proposition}
\newtheorem{remark0}[theorem]{Remark}
\newtheorem{example0}[theorem]{Example}
\newtheorem{definition}[theorem]{Definition}

\newenvironment{example}{\begin{example0}\rm}{\end{example0}}
\newenvironment{remark}{\begin{remark0}\rm}{\end{remark0}}

\newcommand{\defref}[1]{Definition~\ref{#1}}
\newcommand{\propref}[1]{Proposition~\ref{#1}}
\newcommand{\thmref}[1]{Theorem~\ref{#1}}
\newcommand{\lemref}[1]{Lemma~\ref{#1}}

\newcommand{\exref}[1]{Example~\ref{#1}}

\newcommand{\remref}[1]{Remark~\ref{#1}}

\def\max{{\bf m}}                   
\def\res{{\bf k}}                   

\def\ser2{{\res[\![x,y]\!]}}
\def\ser3{{\res[\![x,y,z]\!]}}

\def\con#1{{\mathbb C}\{x_1,\dots ,x_{#1}\}}

\def\ori#1{(\mathbb C^{#1},0)}
\def\Oori[#1]#2{{\mathcal O}_{(\mathbb C^{#1},#2)}}

\def\OX{{\mathcal O}_{(X,0)}}
\def\cont{{\mathbb C}\{t\}}
\def\conu{{\mathbb C}\{u\}}

\def\pol2{{\res[x,y]}}
\def\pol3{{\res[x,y,z]}}

\def\ext{{\mathrm{ Ext}}}

\def\HF{{\operatorname{H\!F}}}
\def\HP{\operatorname{H\!P}}

\def\adj{\operatorname{Adj}}
\def\mr{\operatorname{MatRep}}

\def\Hom{\operatorname{Hom}}
\def\ker{\operatorname{ker}}
\def\coker{\operatorname{coker}}

\def\dim{\operatorname{dim}}

\def\pd{\operatorname{pd}}

\def\length{{\mathrm{ Length}}}

\def\spec{{\bf Spec}}

\title{  \bf Matrix factorization and the generic plane projection of curve a singularity
\footnote{ 2020 {\it Mathematics Subject Classification} 14B05, 13D02, 14B07.
\newline
\indent \ \ {\it Key words and Phrases:} .}}
\author{\large   J. Elias
\thanks{Partially supported by {{PID2022-137283NB-C22}}}
}

\date{ \today}

\begin{document}
\maketitle

\begin{abstract}
We study the matrix factorizations defined by the generic plane projections of a curve singularity of $\ori{n}$.
On the other hand, given a plane curve singularity $(Y,0)\subset \ori{2}$ we study the family of
matrix factorizations defined by the space curve singularities
$(X,0)\subset \ori{n}$ such that $(Y,0)$ is the generic plane projection of $(X,0)$.
\end{abstract}

\section{Introduction}

The description of the structure and the computation of the length of minimal resolutions of  finitely generated modules
are  central problems in several areas of mathematics.
Hilbert's syzygies theorem  and the results of Auslander, Buchsbaum and Serre  prove that regular rings are those for which  all finitely generated modules have finite minimal resolutions.
Koszul complexes and the theorems of Hilbert-Burch and Buschbaum-Eisenbud give  explicit structures
of the (finite) minimal resolutions of complete intersections, codimension two perfect ideals and
codimension three Gorenstein ideals, respectively.
In all these cases the information provided by the minimal resolution is finite and essentially unique.

On the other hand, infinite minimal resolutions naturally arise in algebra since we consider modules over non-regular rings.
See  the contribution  of Avramov in \cite{SCA96} for a good introduction to infinite resolutions.
Eisenbud proved in \cite{Eis80} that finitely generated modules over a complete intersections of codimension one, i.e. $R=S/(F)$ with $R$ a local regular ring,
eventually becomes periodic of period two.
The matrices $A, B$ of the morphisms defining this periodic structure is a matrix factorization of $F$, that means    $AB=BA= F \rm{Id}$.
This construction was generalized to complete intersection rings by Eisenbud and Peeva in \cite{EP16}.
Although the minimal resolution is infinite in such cases, the information defining it remains finite.
See \cite[Chapter 1]{EP16} for a introduction and survey to matrix factorizations,
see also \cite{EP19}, \cite{EPS21} and the references therein.

\medskip
The aim of this paper is to study matrix factorizations that naturally appear when we consider generic plane projections of an irreducible curve singularity of $\ori{n}$.
Given a curve singularity  $(X,0)\subset \ori{n}$ the generic plane projections $(Y,0)$ are well defined up to equisingularity type,
see \defref{gen-proj},
inducing finite ring extensions
$\pi^*: \mathcal O_{(Y,0)}\longrightarrow \mathcal O_{(X,0) }$.
Since $\mathcal O_{(Y,0)}=\mathbb C\{x,y\}/(F)$ is a complete intersection ring,  by \cite{Eis80} there exists a matrix factorization of $F$ defined by a pair of  square matrices $A$, $B$, with entries in $(x,y)\mathbb C\{x,y\}$, such that $AB=BA= F \mathrm{Id}$, inducing an exact sequence
$$
0\longrightarrow
\mathbb C\{x,y\}^b \stackrel{A}{\longrightarrow}
\mathbb C\{x,y\}^b \longrightarrow \mathcal O_{(X,0) }\longrightarrow 0,
$$
and a minimal periodic free resolution of $\mathcal O_{(Y,0)}$-modules
$$
\cdots \longrightarrow
\mathcal O_{(Y,0)}^b \stackrel{B}{\longrightarrow}
\mathcal O_{(Y,0)}^b \stackrel{A}{\longrightarrow}
\mathcal O_{(Y,0)}^b \stackrel{B}{\longrightarrow}
\mathcal O_{(Y,0)}^b \stackrel{A}{\longrightarrow}
\mathcal O_{(Y,0)}^b \longrightarrow \mathcal O_{(X,0) }\longrightarrow 0,
$$
see \thmref{matfact}.
Moreover, we know that the family of generic plane projections is  flat, see section $3$.
This flat family induces a global matrix factorization specializing  to the above matrix factorization,  \thmref{matfact}.

In the section $5$ we consider the family of
matrix factorizations defined by the space curve singularities
$(X,0)\subset \ori{n}$ with a given  a generic plane projection  $(Y,0)\subset \ori{2}$.
In \propref{MatRep-proj}
we characterize the matrix factorizations of $F\in \mathbb C\{x,y\}$, equation defining $(Y,0)$,
induced by  a space curve singularity $(X,0)\subset \ori{n}$.
Equivalently, we characterize the dimension $b$ square matrices $A$ such that the cokernel of
$\mathbb C\{x,y\}^b \stackrel{A}{\longrightarrow}
\mathbb C\{x,y\}^b$ is a curve singularity with generic plane projection $(Y,0)$.

In \thmref{bijection} we establish a bijection between the orbits of the action of the general linear group in the family of matrix factorizations induced by curve singularities with generic plane projection $(Y,0)$ and
the clases of isomorphism of $\mathcal O_{(Y,0)}$-algebras of the rings $\mathcal O_{(X,0)}$,
where $(X,0)\subset \ori{n}$ is a curve singularity with  $(Y,0)\subset \ori{2}$ as a generic plane projection.

The explicit computations of this paper are performed  by using the computer algebra system  Singular  \cite{DGPS}.

\medskip
\section{Preliminaries}

Let $(X,0)$ be a  an analytic reduced curve singularity of $\ori{n}$, from now on  curve singularity.
The local ring $\OX$ is a one-dimensional reduced local ring with maximal ideal $\max_{(X,0)}$.
We write by $I_{(X,0)}\subset {\mathcal O}_{\ori{n}}=\con{n}$ the radical ideal defining  $(X,0)$, i.e.
${\mathcal O}_{(X,0)}=\con{n}/I_{(X,0)}$.

The  Hilbert
function of $(X,0)$ is the numerical function $\HF_{(X,0)}:\mathbb N\longrightarrow \mathbb N$ such that
$\HF_{(X,0)}(n)=\length_{{\mathcal O}_{(X,0)}}(\max_{(X,0)}^n/\max_{(X,0)}^{n+1})$ for all $n\in \mathbb N$,
The  Hilbert-Samuel  function of $(X,0)$ is defined
by
$
\HF^{1}_{(X,0)}(n)=\sum_{j=0}^{n}\HF_{(X,0)}(j)=\length({\mathcal O}_{(X,0)}/\max_{(X,0)}^{n+1}),
$
$n\in \mathbb N$.
It is well known that there exist integers $e_0(X,0), e_1(X,0) \in
\mathbb N$ such that
$
\HP^1_{(X,0)}(T)=e_0(X,0) (T+1) -e_1(X,0)
$
 is the  Hilbert-Samuel polynomial of $(X,0)$,
i.e. $\HF^1_{(X,0)}(n)=\HP^1_{(X,0)}(n)$ for $n \gg 0$.
The integer $e_0(X,0)$
is the multiplicity of $(X,0)$ and
the embedding dimension of $(X,0)$ is $\HF_{(X,0)}(1)= \dim_{\mathbb C}(\max_{(X,0)}  / \max_{(X,0)}^2)$.

Let $\nu: \overline{X}=\spec (\overline{\OX})\longrightarrow (X,0)$ be the normalization of $(X,0)$, where
$\overline{\OX}$ is the integral closure of $\OX$ on its full ring  of fractions ${\mathrm{ tot}}(\OX)$.
The singularity order of $(X,0)$ is
$
\delta(X,0)=\dim_{\mathbb C}\left(\overline{\OX}/\OX\right).
$
We denote by $\mathcal C(X,0)$ the conductor of the finite extension
$\nu^*: \OX \hookrightarrow
\overline{\OX}$
and by $c(X,0)$ the dimension of $\overline{\OX}/\mathcal C(X,0)$.
Let  $\omega_{(X,0)}=\ext ^{n-1}_{\Oori[n]{0}}(\OX, \Omega^n_{\ori{n}})$ be the dualizing module of $(X,0)$.
We can consider the composition morphism of $\OX$-modules
$$
\gamma_X: \Omega_{(X,0)} \longrightarrow
\nu_* \Omega_{\overline{X}}\cong
\nu_* \omega_{\overline{X}} \longrightarrow
 \omega_{(X,0)}.
$$
Let $d: \OX \longrightarrow  \Omega_{(X,0)}$ the universal derivation.
The Milnor number of $(X,0)$ is, \cite{BG80},
$$\mu(X,0)=\dim_{\mathbb C}(\omega_{(X,0)}/(\gamma_X \circ d) \OX)$$
Notice that $(X,0)$ is non-singular iff $\mu(X,0)=0$ iff $\delta(X,0)=0$ iff $c(X,0)=0$.

\medskip
In the following result we collect some basic results on $\mu(X,0)$ and other numerical invariants that
we will use later on.

\medskip
\begin{proposition}
\label{basic}
Let $(X,0)$ be a  curve singularity of embedding dimension $n$.
Then

\noindent
$(i)$ $\mu(X,0)= 2 \delta(X,0) -r(X,0)+1$, where $r(X,0)$ is the number of branches of $(X,0)$.

\noindent
$(ii)$
It holds
$$
 e_0(X,0)-1\le e_1(X,0) \le \delta(X,0)\le \mu(X,0)
$$
and $e_1(X,0)\le \binom{e_0(X,0)}{2}-\binom{n-1}{2}$.

\noindent
$(iii)$ If $X$ is singular then $\delta(X,0)+1 \le c(X,0)\le 2 \delta(X,0)$, and $c(X,0) =  2 \delta(X,0)$
if and only if $\OX$ is a Gorenstein ring.
\end{proposition}
\begin{proof}
$(i)$  \cite[Proposition 1.2.1]{BG80}.
$(ii)$ \cite[Proposition 1.2.4 (i)]{BG80}, \cite{Nor59a},  \cite{Eli90}, \cite{Eli01}.
$(iii)$ \cite[Proposition 7, pag. 80]{Ser59}, and \cite{BC77}.
\end{proof}

\medskip
Given a local ring $(A,\max)$ and a finitely generated $A$-module $M$ we denote by
$\beta_A(M)$ the minimal number of generators of $M$, i.e.
$\beta_A(M)=\dim_{A/\max}(M/\max M)$.


\medskip
\section{On the generic projection of a curve singularity}

Throughout this paper we only consider irreducible curve singularities.

Let $(X,0)$ be a curve singularity of $\ori{n}$, i.e. $(X,0)$ is an irreducible curve singularity of $\ori{n}$.
Then $\overline{\OX}\cong \cont$ for some uniformation parameter $t\in {\mathrm{ tot}}(\OX)$ and the morphism
$\nu^*: \OX \hookrightarrow
\overline{\OX}\cong \cont$ can be described
by a family of  series $\nu^*(\overline{x_i})=f_i(t)\in \cont$, $i=1,\dots,n$.
We say that
$$
(X,0) : \left\{
  \begin{array}{l}
    x_1=f_1(t) \\
 \quad\quad \vdots \\
x_n=f_n(t) \\
  \end{array}
\right.
$$
is a parametrization of $(X,0)$.
Notice that the $\mathbb C$-algebra morphism defined by
$$
\begin{array}{cccc}
    \gamma:&\con{n} &  \longrightarrow &\cont\\
    &x_i & \mapsto & f_i(t)
\end{array}
$$
has as kernel the ideal $I_{(X,0)}$ and induces the normalization map $\nu^*$.
The conductor of this ring extension is principal
$\mathcal C(X,0)=(t^{c(X,0)})$.
After a linear change of variables and choosing an appropriate uniformation parameter, i.e. an appropriate generator of the maximal ideal of $\cont$  , we may assume that  the parametrization of $(X,0)$ is defined by:
$f_1(t)=t^{e_0(X,0)}$ and
$f_i(t)=\sum_{j\ge e_0(X,0)+1} b_{j,i} t^j$, $i=2,\dots,n$, with $b_{j,i}\in\mathbb C$.
We say that such a parametrization is presented in a standard form.

We denote by $\pi:Bl(X,0)\longrightarrow (X,0)$ the blowing-up of $(X,0)$ on its closed point.
The fiber of the closed point of $(X,0)$ has a finite number of closed points: the so-called points of the first neighborhood of $(X,0)$.
We can iterate the process of blowing-up until we get the normalization of $(X,0)$, see  \cite{Cut-ResSing}.
We denote by $\inf(X,0)$ the set of infinitely near points of $(X,0)$.
The curve singularity defined by an infinitely point $p$ of $(X,0)$ will be denote by $(X,p)$.
From \cite{Nor59a} we get:

\medskip
\begin{proposition}
\label{basic2}
Let $(X,0)$ be a curve singularity, then
$$
\delta(X,0)=\sum_{p \in \inf(X,0)} e_1(X,p)
$$
\end{proposition}

\noindent
See \cite{Hir57} for the corresponding result for  non-irreducible curve singularities.

\medskip
We denote by $\Gamma_{(X,0)}\subset \mathbb N$ the semigroup of values of $(X,0)$: the set of integers  $v_t(f)=ord_t(t)$ where $f\in  \OX\setminus \{0\}$.
It is easy to see that $\delta(X,0)=\#(\mathbb N\setminus \Gamma_{(X,0)})$.
The multiplicity of $(X,0)$ is
$e_0(X,0)=\min\{v_t(f); f\in \max_X\setminus \{0\}\}$.

\medskip
Next, we recall the definition of the  Whitney's  $C_5(X,0)$ cone, see \cite[Section 3]{Whi65} and  \cite[Chapter IV]{BGG80},

\medskip
\begin{definition}
Let $(X,0)$ be a curve singularity of $\ori{n}$ with a parametrization  $x_i=f_i(t)$, $i=1,\dots, n$.
The cone $C_5(X,0)$ is the set of secant vectors to $(X,0)$: the set of vectors $w\in \mathbb C^n$ such that there exist two analytic functions $\alpha(u), \beta(u)\in \conu$ such that:
$$
w=\lim_{u\rightarrow 0}(f_i(\alpha(u))-f_i(\beta(u)))_{i=1,\dots,n}.
$$
\end{definition}

In \cite[Proposition IV.1]{BGG80}  the cone $C_5(X,0)$ is computed for a general reduced curve singularity.
They show that such a cone is a finite union of dimension two planes containing the tangent lines of $(X,0)$.
Next, we will explicitly describe $C_5(X,0)$ for a irreducible curve singularity  with a
parametrization presented in a standard form.

\medskip
\begin{proposition}
\label{cone5}
Let $(X,0)$ be a curve singularity of $\ori{n}$ with a standard parametrization, $e=e_0(X,0)$,
$$
\left\{
  \begin{array}{ll}
    x_1=t^e& \\
x_i= \sum_{j\ge e+1} b_{j,i} t^j& i=2,\dots,n\\
  \end{array}
\right.
$$
The cone $C_5(X,0)$ is the union of the dimension two planes
generated by $(1,0,\dots,0)\in\mathbb C^n$ and the vectors
$(0,b_{k,2},\dots,b_{k,n})\in \mathbb C^n\setminus\{0\}$ for which there exist $\epsilon\in\mathbb C$
with $\epsilon^e=1$ and $k$ is the first integer $j\ge e+1$, if it exists, for which
$ (\epsilon^{j}-1)b_{j,i}\neq 0$ for some integer $i\in\{2,\dots,n\}$.
\end{proposition}
\begin{proof}
See the proof of \cite[Proposition IV.1]{BGG80}.
\end{proof}

\medskip
\begin{remark}
Notice that from the  last result we get that the number of the irreducible components of $C_5(X,0)$ is at most $e_0(X,0)-1$.
\end{remark}

Given a set of non-negative integers $2\le a_1<\cdots <a_n$ such that $gcd(a_1,\cdots, a_n)=1$,
we consider the monomial curve singularity $(X,0)=M(a_1,\cdots, a_n)$ defined by the $\mathbb C$-algebra morphism
$\gamma: \con{n}\longrightarrow \mathbb C\{t\}$ with $\gamma(x_i)=t^{a_i}$, $i=1,\dots,n$,
i.e. $I_{(M(a_1,\cdots, a_n),0)}=\ker(\gamma)$.
The induced monomorphism
$$
\gamma:\OX=\con{n}/I_{M(a_1,\cdots, a_n)}\longrightarrow \cont
$$
is the normalization map.

\medskip
Next, we compute Whitney's  $C_5$  cone of  a monomial curve.

\medskip
\begin{proposition}
Given a monomial curve singularity $(X,0)=M(n_1,n_2,n_3)$ of $\ori{3}$.
\begin{enumerate}
\item[(i)] If $\gcd(n_1,n_2)=1$ then $C_5(X,0)$ is the plane generated by $(1,0,0)$ and $(0,1,0)$.
\item[(ii)]
If $\gcd(n_1,n_2)=a>1 $ then $C_5(X,0)$ is the union of the plane generated by $(1,0,0)$ and $(0,1,0)$ and the plane generated by
$(1,0,0)$ and $(0,0,1)$.
\end{enumerate}
\end{proposition}
\begin{proof}
If $\gcd(n_1,n_2)=1$ then from \propref{cone5} we get that $C_5(X,0)$ is the plane generated by $(1,0,0)$ and $(0,1,0)$.
Assume now that $\gcd(n_1,n_2)=a>1$.
Let $\xi$ be a $n_1$-th unit root.
If $\xi^a\neq 1$ then we deduce that  $C_5(X,0)$ contains the plane generated by $(1,0,0)$ and $(0,1,0)$.
Assume that $\xi^a=1$; since $\gcd(a,n_3)\neq1$ from \propref{cone5} we get that $C_5(X,0)$ contains the plane generated by $(1,0,0)$ and $(0,0,1)$.
\end{proof}

\medskip
\begin{example}
\label{exC5}
Let us consider the monomial curve $(X,0)$ of $\ori{3}$
defined by the standard parametrization $x_1=t^4,  x_2=t^6, x_3=t^7$.
We know that the defining ideal of $(X,0)$ is $I_{(X,0)}=(x_1^3-x_2^2,x_3^2-x_1^2x_2)$.
We have $e_0(X,0)=4$ and the $4$-th unit roots are: $\{1,-1,\imath, -\imath \}$.
Then $C_5(X,0)$ is the cone union of the two $2$-dimensional planes generated by
$(1,0,0)$, defining the  tangent line of $(X,0)$,  and $(0,1,0)$, for  $\epsilon=\imath, -\imath$, and the plane generated by
$(1,0,0)$ and $(0,0,1)$, for  $\epsilon=-1$.
\end{example}

\medskip
We denote by $\mathcal G(n,n-2)$ the Grassmannian of the dimension $n-2$ linear varieties of $\ori{n}$.
For all $H\in \mathcal G(n,n-2)$ we denote by $\pi_H$ the projection to $\ori{2}$ parallel to $H$.
We denote by $\mathcal W_{(X,0)}$ the open set  of $H\in \mathcal G(n,n-2)$ such that $H\cap C_5(X,0)=\{0\}$.

Next, we recall \cite[Proposition IV.1]{BGG80} and we sketch the proof that we will use later on.
For the theory of equisingularity for plane curve curves see \cite{Zar65c}, or  \cite{BGG80} and \cite{GLS07}.

\medskip
\begin{proposition}
\label{generic-proj}
For all $H\in \mathcal W_{(X,0)}$ the image $\pi_H(X,0)$ is a germ of  a reduced curve singularity with constant equisingularity type.
\end{proposition}

Next, we recall how this result was proved; see  \cite[Proposition IV.1]{BGG80} for the details.
We denote by $W=W_{(X,0)}$ the Zariski open subset of $\mathbb C^{2n}$, with coordinates $z_1,\dots , z_{2n}$, parameterizing the pairs of linear forms
$(L_1,L_2)$ with $L_1=z_1 x_1+\cdots +z_nx_n$,
$L_2=z_{n+1} x_1+\cdots +z_{2n}x_n$
such that the linear variety defined by them is of dimension $n-2$ and belonging to  $\mathcal W_{(X,0)}$.
For each $L=(z_1,\dots,z_{2n})\in W$ we denote by
$$
\begin{array}{ccccc}
  \pi_L: & \mathbb C^n & \longrightarrow  & \mathbb C^2 \\
         &(x_1,\dots, x_n)     &  \mapsto         & (L_1,L_2)
\end{array}
$$
the projection defined by the dimension $n-2$ linear variety $L_1=L_2=0$.
The image of $(X,0)$ is a  plane curve singularity $\pi_L(X,0)\subset \ori{2}$.
We denote by $\pi$ the  induced morphism
$$
\begin{array}{ccccc}
  \pi: & \mathbb C^n \times W & \longrightarrow  & \mathbb C^2 \times W\\
         &((x_1,\dots, x_n), L)     &  \mapsto         & (\pi_L(x_1,\dots, x_n),L)
\end{array}
$$

\noindent
Given $L=(c_1,\dots,c_{2n})\in W$ the restriction of $\pi$ to $(X,0)\times W$ if a finite morphism in $(0,L)$, see proof of \cite[Proposition IV.1]{BGG80}.
The image
$\pi((X,0)\times W,(0,L))$ is a hypersurface  of $\ori{2}\times (W,L)$
defined by a principal ideal
$(F(x,y;z_1-c_1,\dots,z_{2n}-c_{2n}))\subset  \mathbb C\{x,y;z_1-c_1,\dots,z_{2n}-c_{2n}\}
\cong \mathcal O_{\ori{2}\times (W,L)}$.
Then we have a $\mathbb C$-algebra monomorphism
\begin{equation*}
\label{eqs}
\mathcal O_{\pi((X,0)\times W,(0,L)),(0,L)}\cong
\frac{\mathbb C\{x,y;z_1-c_1,\dots,z_{2n}-c_{2n}\}}{(F(x,y;z_1-c_1,\dots,z_{2n}-c_{2n}))}\stackrel{\pi^*}{\longrightarrow}
\mathcal O_{X\times W, (0,L)}
\end{equation*}
such that tensoring by \quad
$\cdot \otimes_{\mathbb C\{x,y;z_1-c_1,\dots,z_{2n}-c_{2n}\}}\frac{\mathbb C\{x,y;z_1-c_1,\dots,z_{2n}-c_{2n}\}}{(z_1-c_1,\dots,z_{2n}-c_{2n})}$
we get the morphism
 \begin{equation*}
\label{eqsL}
\mathcal O_{\pi_L(X,0)}
\stackrel{\pi^*_L}{\longrightarrow}
\mathcal O_{X,0},
\end{equation*}
 morphism induced by the projection
$
(X,0)\stackrel{\pi_L}{\longrightarrow}  \pi_L(X,0).
$
Moreover,
the  projection into the second component
$$
\gamma: \pi((X,0)\times W,(0,L)) \longrightarrow (W,L)
$$
defines
a flat deformation of $\pi_L(X,0)$ whose fibers are curve singularities with constant Milnor number.
This flat deformation $\gamma$ admits a resolution in family,
this means that  if
$\nu: \overline{\pi((X,0)\times W,(0,L))}\longrightarrow \pi((X,0)\times W,(0,L))$ is the normalization map
then the composition
$$
G= \gamma \circ \nu: \overline{\pi((X,0)\times W,(0,L))}
\stackrel{\nu}{\longrightarrow}\pi((X,0)\times W,(0,L))
\stackrel{\gamma}{\longrightarrow} (W,L)
$$
satisfies
$G^{-1}(p)=\overline{\pi_p(X,0)}$
for a $p$ belonging to a small neighbourhood of $L$ in $W$.
Given a parametrization
$$
\left\{
 \begin{array}{ll}
x_1=f_1(t)&\\
\vdots & \\
x_n=f_n(t)
 \end{array}
\right.
$$
of $(X,0)$ the flat deformation $\gamma$
admits a parametrization in family.
This means that the normalization map
$\overline{\pi_p(X,0)}\longrightarrow \pi_p(X,0) $ is defined, after a change of variables, by
$$
\left\{
  \begin{array}{ll}
x=\sum_{j=1}^{n} (s_j+c_j) f_j(t)& \\ \\
y=\sum_{j=n+1}^{2n} (s_{j}+c_j) f_{j-n}(t)\\
 \end{array}
\right.
$$
for small $s_1,\dots ,s_{2n}$, \cite{Tei77}, \cite{Tei80a}.
Furthermore, we have a commutative  diagram where $\nu$ and $\nu_X$ are the corresponding normalization morphisms
\begin{center}
\begin{tikzcd}[column sep=large]
\mathcal O_{(X,0)\times W,(0,L)} \arrow[ r,hook,  "(\nu_X\times Id_W)^*"]  &\overline{\mathcal O_{(X,0)\times W,(0,L)} }\cong\overline{\mathcal O_{\pi((X,0)\times W,(0,L)),(0,L)}}  \\
\mathcal O_{\pi((X,0)\times W,(0,L)),(0,L)} \arrow[u, hook, "\pi^*"'] \arrow[ ur, hook, "\nu^*"]&
\end{tikzcd}
\end{center}
notice that $\nu^*_X$ is defined by the given parametrization of $(X,0)$ and  $\nu^*$ is defined by above parametrization in family.

\medskip
\begin{definition}
\label{gen-proj}
A generic plane projection  of a curve singularity $(X,0)$ is $\pi_H(X,0)$ for any $H\in \mathcal W_{(X,0)}$.
All these plane curve singularities share the same equisingularity type, \propref{generic-proj}.
We denote by $\overline{\mu}(X,0)$ the Milnor number of a generic plane projection of $(X,0)$.
\end{definition}

\medskip
\begin{remark}
A classical problem considered by F. Enriques and O. Chisini was to compare the equisingularity type of $(X,0)$ and the equisingularity type of a generic plane projection $(\pi_L(X,0),0)$, see \cite[page 12]{Zar71}.
This problem was addressed  in \cite{Cas78} and  \cite{CC82}.
On the other hand, in \cite{Eli88}  we  gave a sharp upper bound of the singularity order of  $(\pi_L(X,0),0)$, $L$ generic, in terms of the  singularity order of $(X,0)$,
we proved that
$$
\delta(X,0)\le \delta(\pi_L(X,0),0)\le
(e_0(X,0)-1)\delta(X,0)-\binom{e_0(X,0)-1}{2}.
$$
\end{remark}

\medskip
\begin{example}
\label{exC5def}
Let us consider \exref{exC5}.
We take coordinates $z_1,z_2,z_3,z_4,z_5,z_6$  in $\mathbb C^6$ parameterizing
the pairs of linear forms of $\mathbb C\{x_1,x_2,x_3\}$.
In this case we have that the Zariski open set
$W_{(X,0)}\subset \mathbb C^{6}$ is defined by the conditions $z_1z_6-z_3z_4\neq 0$ and $z_1z_5-z_2z_4\neq 0$.
Next, we study $\pi_L(X,0)$ in a neighbourhood of $L=(1,0,0,0,1,1)\in W_{(X,0)}$.
The parametrization in family is
$$
\left\{
  \begin{array}{ll}
x=  (1+s_1) t^4 + s_2 t^6+ s_3 t^7& \\
y=   s_4 t^4+ (1+s_5) t^6+ (1+s_6)t^7.&\\
  \end{array}
\right.
$$
for small $s_1,\dots ,s_{2n}\in \mathbb C$.
In order to get "small"  equations we consider the sub-family defined by $s_1=s_2=s_3=s_4=s_5=0$
$$
\left\{
  \begin{array}{ll}
x=   t^4& \\
y=   t^6+ (1+s_6)t^7&\\
  \end{array}
\right.
$$
for a small $s_6$.
We can compute the equation $F$  eliminating the variable $t$ of the ideal
$(x- t^4, y-t^6- (1+s_6)t^7)$,\cite{DGPS},
\begin{eqnarray*}
F&=& y^4-2x^3y^2+y^4s_6+x^6-4x^5y-2x^3y^2s_6-x^7+x^6s_6-12x^5ys_6-5x^7s_6
 \\
 && -12x^5ys_6^2-10x^7s_6^2-4x^5ys_6^3-10x^7s_6^3-5x^7s_6^4-x^7s_6^5.
\end{eqnarray*}
Notice that, in particular,
$F(x,y,0)=0$ is the equation of the ideal of the monomial curve singularity $(X,0)$.
Hence we have a commutative diagram
\begin{center}
\begin{tikzcd}[column sep=large]
\frac{\mathbb C\{x_1,\dots,x_n,s_6\}}{I_{(X,0)}\mathbb C\{x_1,\dots,x_n,s_6\}}\arrow[ r,hook, "(\nu_X\times Id_W)^*"]  &  \mathbb C\{t,s_6\}\\
\frac{\mathbb C\{x,y,s_6\}}{(F)} \arrow[u, hook, "\pi^*"'] \arrow[ ur,hook, "\nu^*"]&
\end{tikzcd}
\end{center}
where $\nu_X^*(x_1)=t^4$,  $\nu_X^*(x_2)=t^6$, $\nu_X^*(x_3)=t^7$, and $\nu^*(x)=t^4$ and $\nu^*(y)=t^6+(1+s_6)t^7$.

Since the fibers of $\gamma$ are plane curves  we can compute their Milnor number as  follows.
Provided $s_6$ small the multiplicity sequence of the resolution process of the fiber $\gamma^{ -1}(s_6)$
is $\{4,2,2,1,\dots\}$, so $\mu(\gamma^{ -1}(s_6),0)=16$,  \propref{basic} (iii) and \propref{basic2}.
Hence $\overline{\mu}(X,0)=16$.
\end{example}

\medskip
\begin{example}
\label{anaType}
In \propref{generic-proj} we recall that all generic plane projections share the same equisingularity type.
In the following example we show that the analytic type of the generic plane projections are not constant.
We consider the example $4$ of \cite[Chapter V]{Zar06},
$(X,0)=M(5,6,8,9)\subset \ori{4}$.
By using \propref{cone5} we get that $\mathcal W_{(X,0)}$ is the set of planes through the origin of $\mathbb C^4$ transversal to $x_2=x_3=0$.
Let us consider the plane $L(a_8,a_9)$ defined by the linear forms $x_1$ and $x_2+a_8 x_3 + a_9 x_4$, with $a_8\neq 0$ and $a_9\neq 0$, we have that
$L(a_8,a_9)\in \mathcal W_{(X,0)}$.
The projection $\pi_{L(a_8,a_9)}(X,0)$ has the following parametrization
$$
\left\{
 \begin{array}{ll}
x=t^5&\\
y=t^6+a_8t^8+a_9t^9.&
 \end{array}
\right.
$$
From \cite[Proposition 4.1, Chapter V]{Zar06}
we have that two of such projections with respect to $L(a_8,a_9)$ and $L(b_8,b_9)$ are analytically isomorphic if and only if
$a_8^3/a_9^2=b_8^3/b_9^2$.
Hence there are infinitely many analytic isomorphism classes of generic plane projections of $(X,0)$.
The equisingularity type of these curve singularities is constant and  it is defined by the set of characteristic exponents
$\{6/5, 8/5, 9/5\}$, \cite{Zar65c} or \cite{BGG80}.
\end{example}

\medskip
\section{Matrix decomposition and the generic plane projection}

In this section we study the family of matrix factorizations naturally  appearing when we consider the generic plane projections of a curve singularity, \cite{Eis80}, \cite{EP16}.
The key ingredient is that the generic plane projection define a flat family with constant equisingularity type, see proof of \propref{generic-proj}.
We construct a "family" of matrix factorizations with closed fiber a matrix factorization of the initial germ of a curve singularity.

As a motivation we start describing  a concrete example of matrix factorization.
We construct a matrix factorization of the monomial curve singularity of \exref{exC5}.

\medskip
\begin{example}
\label{exC5MF}
Let us consider the ring $\mathcal O_{(X,0)}=\mathbb C\{x_1,x_2,x_3\}/I_{(X,0)}$, \exref{exC5}.
After the $\mathbb C$-algebra isomorphism  of $\mathbb C\{x_1,x_2,x_3\}$
$$
\phi:
\left\{
  \begin{array}{ll}
   \phi(x_1)&=x \\
    \phi(x_2)&=y-z \\
\phi(x_3)&=z
  \end{array}
\right.
$$
we get that $\mathcal O_{(X,0)}\cong M_X:=\mathbb C\{x,y,z\}/J$
where $J$ is the ideal generated by
$x^3-(y-z)^2$ and $z^2-x^2(y-z)$.
The $\mathbb C\{x,y\}$-module $M_X$ is generated by $1$ and $\overline{z}$, so we have a complex of
$\mathbb C\{x,y\}$-modules
$$
0\longrightarrow
\mathbb C\{x,y\}^2 \stackrel{d}{\longrightarrow}
\mathbb C\{x,y\}^2 \stackrel{(1,\overline{z})}{\longrightarrow} M_X\longrightarrow 0
$$
where
$$
d=
\left(
  \begin{array}{cc}
    x^3-y^2-x^2y & x^4y+2x^2y^2 \\
    x^2+2y       & x^3-x^4-3x^2y-y^2 \\
  \end{array}
\right).
$$
With the help of Singular we can show that this complex is exact, \cite{DGPS}.
Moreover, let us consider the morphism
$h:
\mathbb C\{x,y\}^2 \longrightarrow
\mathbb C\{x,y\}^2
$
defined by  the adjoint matrix of  $d$.
Then we can check that
$$
h d= d h=F \cdot \mathrm{Id}_{\mathbb C\{x,y\}^2}
$$
where
$F=y^4-2x^3y^2+x^6-4x^5y-x^7$
is the equation of the plane projection
$$
\pi_L(X,0):
\left\{
  \begin{array}{l}
    x =t^4 \\
    y =t^6+t^7
  \end{array}
\right.
$$
with $L=(1,0,0,0,1,1)\in W_{(X,0)}$.
Hence the pair $(d,h)$ is a matrix factorization of $F$, see  \cite[Section 1.2]{EP16}.
Then we have a minimal periodic two free resolution of the
$\mathcal O_{\pi_L(X,0)}=\mathbb C \{x,y\}/(F)$-module $M_X$
\begin{multline*}
\cdots \longrightarrow
\mathcal O_{(\pi_L(X,0),0)}^2 \stackrel{h}{\longrightarrow}
\mathcal O_{(\pi_L(X,0),0)}^2 \stackrel{d}{\longrightarrow}
\mathcal O_{(\pi_L(X,0),0)}^2 \stackrel{h}{\longrightarrow}\\
\stackrel{h}{\longrightarrow}\mathcal O_{(\pi_L(X,0),0)}^2 \stackrel{d}{\longrightarrow}
\mathcal O_{(\pi_L(X,0),0)}^2 \stackrel{(1,\overline{z})}{\longrightarrow} M_X\longrightarrow 0,
\end{multline*}
\cite[Theorem 2.1.1]{EP16}.
\end{example}

\medskip
In the next statement we use  the notations of the proof of \propref{generic-proj}.
For all $L\in W_{(X,0)}$ we consider $\mathcal O_{X\times W,(0,L)}$ as a
$\mathcal O_{\pi(X\times W,(0,L)),(0,L)}$-module via $\pi^*$, and
 $\mathcal O_{(X,0)}$ as
$\mathcal O_{(\pi_L(X,0),0)}$-module via $\pi_L^*$.

\medskip
\begin{theorem}
\label{matfact}
Let $(X,0)$ be a curve singularity of $\ori{n}$.
Let $b$ the minimal number of generators of $\mathcal O_{(X,0)}$ as
$\mathcal O_{(\pi_L(X,0),0)}$-module.
Then for all  $L=(c_1,\dots,c_{2n})\in W_{(X,0)}$ there is
$\widehat{F}=F(x,y;z_1-c_1,\dots,z_{2n}-c_{2n})\in R=\mathbb C\{x,y;z_1-c_1,\dots,z_{2n}-c_{2n}\}$
such that
 $\mathcal O=\mathcal O_{\pi((X,0)\times W,(0,L)),(0,L)}\cong \frac{R}{(\widehat{F})}$
and a matrix factorization of $\widehat{F}$
$$
\widehat{d}\; \; \widehat{h}=  \widehat{h}\; \widehat{d}=\widehat{F} \cdot \mathrm{Id}_{R^b}
$$
with
$\widehat{d}, \widehat{h}\in \mathrm{Mat}_{b\times b}(R)$ such that
$$
0 \longrightarrow
R^b\stackrel{\widehat{d}}{\longrightarrow}
R^b\longrightarrow
\mathcal O_{(X,0)\times W,(0,L)}\longrightarrow 0
$$
is exact.
Moreover,
there is a minimal free resolution of $\mathcal O_{(X,0)\times W,(0,L)}$ as
$\mathcal O$-module
$$
 \cdots \longrightarrow
\mathcal O^b \stackrel{\widehat{h}}{\longrightarrow}
\mathcal O^b \stackrel{\widehat{d}}{\longrightarrow}
\mathcal O^b \stackrel{\widehat{h}}{\longrightarrow}
\mathcal O^b \stackrel{\widehat{d}}{\longrightarrow}
\mathcal O^b \longrightarrow
\mathcal O_{(X,0)\times W,(0,L)}
\longrightarrow 0
$$
such that tensoring by
$ \cdot \otimes_{R} \frac{R}{(z_1-c_1,\dots,z_{2n}-c_{2n})}$
we get a minimal free resolution of
$\mathcal O_{(\pi_L(X,0),0)}$-modules
\begin{multline*}
 \cdots \longrightarrow
\mathcal O_{(\pi_L(X,0),0)}^b \stackrel{h}{\longrightarrow}
\mathcal O_{(\pi_L(X,0),0)}^b \stackrel{d}{\longrightarrow}
\mathcal O_{(\pi_L(X,0),0)}^b \stackrel{h}{\longrightarrow}\\
 \stackrel{h}{\longrightarrow}\mathcal O_{(\pi_L(X,0),0)}^b \stackrel{d}{\longrightarrow}
\mathcal O_{(\pi_L(X,0),0)}^b \longrightarrow
\mathcal O_{(X,0)}
\longrightarrow 0
\end{multline*}
with $h, d$ the induced morphisms by $\widehat{d}, \widehat{h}$.
In particular, the pair $d,h$ is a matrix factorization of $F(x,y;0,\dots,0)$, equation of
$(\pi_L(X,0),0)$.
\end{theorem}
\begin{proof}
Let $b$ the minimal number of generators of $\mathcal O_{(X,0)}$ as
$\mathcal O_{(\pi_L(X,0),0)}$-module.
Then the minimal number of generators of $\mathcal O_{X\times W,(0,L)}$
as $\mathcal O$-module is $b$ as well.
Since $\mathcal O$ is a quotient of $R$, the minimal number of generators of $\mathcal O_{(X,0)\times W,(0,L)}$ as $R$-module is  $b$.

Since $\mathcal O_{(X,0)}$ is a finite and Cohen-Macaulay $\mathcal O_{(\pi_L(X,0),0)}$-module of dimension one,
$\mathcal O_{(X,0)}$ is a projective dimension one $\mathbb C\{x,y\}$-module.
On the other hand, since $\{z_1-c_1,\dots,z_n-c_n\}$ is a regular sequence of
$R$
$$
\pd_R(\mathcal O_{(X,0)\times W,(0,L)})=\pd_{\mathbb C\{x,y\}}(\mathcal O_{(X,0)})=1.
$$

Assume that there exists a decomposition of
$\mathcal O$-modules with $r\ge 1$
$$
\mathcal O_{(X,0)\times W,(0,L)}\cong \mathcal O^r \oplus M,
$$
tensoring by
$ \cdot \otimes_{R} \frac{R}{(z_1-c_1,\dots,z_{2n}-c_{2n})}$
we get the following decomposition of $\mathcal O_{(\pi_L(X,0),0)}$-modules
$$
\mathcal O_{(X,0)}\cong \mathcal O_{(\pi_L(X,0),0)}^r\oplus \overline{M}.
$$
Since $\mathcal O_{(X,0)}$ is a finite length
$\mathcal O_{(\pi_L(X,0),0)}$-module we get a contradiction.
Then
$\mathcal O_{(X,0)\times W,(0,L)}$
is an $\mathcal O$-module with no free summands.

Recall that from the proof of \propref{generic-proj} there exists
$\widehat{F}=F(x,y;z_1-c_1,\dots,z_{2n}-c_{2n})\in R=\mathbb C\{x,y;z_1-c_1,\dots,z_{2n}-c_{2n}\}$
such that
 $\mathcal O\cong \frac{R}{(\widehat{F})}$.
Being  $\mathcal O_{(X,0)\times W,(0,L)}$  a maximal Cohen-Macaulay $\mathcal O_{\pi((X,0)\times W,(0,L)),(0,L)}$-module with no free summand, from
\cite[Theorem 6.1 (ii)]{Eis80} there exist a matrix factorization of $\widehat{F}$
$$
\widehat{d}\; \; \widehat{h}=  \widehat{h}\; \widehat{d}=\widehat{F} \cdot \mathrm{Id}_{R^b}
$$
with
$\widehat{d}, \widehat{h}\in \mathrm{Mat}_{b\times b}(R)$
defining  a minimal free resolution of $\mathcal O_{(X,0)\times W,(0,L)}$ as
$\mathcal O$-module
$$
 \cdots \longrightarrow
\mathcal O^b \stackrel{\widehat{h}}{\longrightarrow}
\mathcal O^b \stackrel{\widehat{d}}{\longrightarrow}
\mathcal O^b \stackrel{\widehat{h}}{\longrightarrow}
\mathcal O^b \stackrel{\widehat{d}}{\longrightarrow}
\mathcal O^b \longrightarrow
\mathcal O_{(X,0)\times W,(0,L)}
\longrightarrow 0.
$$
From \cite[Theorem 2.1.1 (1)]{EP16} we get a minimal resolution of $R$-modules
$$
0 \longrightarrow
R^b\stackrel{\widehat{d}}{\longrightarrow}
R^b\longrightarrow
\mathcal O_{(X,0)\times W,(0,L)}\longrightarrow 0.
$$
Tensoring this resolution by $ \cdot\otimes_{R}
\frac{R}{(z_1-c_1,\dots,z_{2n}-c_{2n})}$ we get a resolution
$$
0 \longrightarrow
\mathbb C\{x,y\}^b\stackrel{d}{\longrightarrow}
\mathbb C\{x,y\}^b\longrightarrow
\mathcal O_{(X,0)}
\longrightarrow 0
$$
with $d$ the induced morphism.
Notice that the matrix factorization of $\hat{F}$ induces  a matrix factorization of $F(x,y;0,\dots,0)$
$$
d\;  h= h\; d=F \cdot \mathrm{Id}_{\; \mathbb C\{x,y\}^b}
$$
with
$d, h\in \mathrm{Mat}_{b\times b}(\mathbb C\{x,y\})$ the induced morphisms by $\widehat{d}, \widehat{h}$.
From\cite[Theorem 1.2.1 (2)]{EP16} we get a minimal free resolution of
$\mathcal O_{(\pi_L(X,0),0)}$-modules
\begin{multline*}
 \cdots \longrightarrow
\mathcal O_{(\pi_L(X,0),0)}^b \stackrel{h}{\longrightarrow}
\mathcal O_{(\pi_L(X,0),0)}^b \stackrel{d}{\longrightarrow}
\mathcal O_{(\pi_L(X,0),0)}^b \stackrel{h}{\longrightarrow}\\
 \stackrel{h}{\longrightarrow}\mathcal O_{(\pi_L(X,0),0)}^b \stackrel{d}{\longrightarrow}
\mathcal O_{(\pi_L(X,0),0)}^b \longrightarrow
\mathcal O_{(X,0)}
\longrightarrow 0
\end{multline*}
\end{proof}

\medskip
In the next example we explicitly compute the resolutions and matrix factorizations appearing in the last result for the monomial curve singularity of \exref{exC5def}.

\medskip
\begin{example}
\label{exC5MF-def}
Let $(X,0)$ be the curve singularity of \exref{exC5def}.
Let us consider the ring
$$
\mathcal O_{((X,0)\times W,(0,L))}\cong\frac{R}{R\; I_{(X,0)}}
$$
where $R=\mathbb C\{x_1,x_2,x_3,s\}$.
After the $\mathbb C$-algebra isomorphism  of $R$
$$
\phi:
\left\{
  \begin{array}{ll}
   \phi(x_1)&=x \\
    \phi(x_2)&=y-(1+s)z \\
\phi(x_3)&=z\\
\phi(s)&=s
  \end{array}
\right.
$$
we get that $\mathcal O_{((X,0)\times W,(0,L))}\cong M_X:=\mathbb C\{x,y,z,s\}/J$
where $J$ is the ideal generated by
$x^3-(y-(1+s)z)^2,z^2-x^2(y-(1+s)z)$.
We have that
$
\mathcal O_{(\pi((X,0)\times W),(0,L)),(0,L)}\cong \frac{R}{(\widehat{F})}
$
, $L=(1,0,0,0,1,1)$, with
$$
{\widehat F}=y^4-2x^3y^2+x^6-4x^5y-x^7-s^4 x^7-4 s^3 x^7-6 s^2 x^7-4 s^2 x^5 y-4 s x^7-8 s x^5 y
$$
The $\mathbb C\{x,y,s\}$-module $M_X$ is generated by $1, \overline{z}$ and we have a matrix factorization of ${\widehat F}$
with associated matrices
{\scriptsize{
$$
{\widehat d}=\left(
\begin{array}{cc}
 x^3-y^2-(s+1)^2 x^2 y  & x^4 y+2 x^2 y^2+s^3 x^4 y+3 s^2 x^4 y+3 s x^4 y+2 s x^2 y^2 \\
 s (s+1)^2 x^2+(s+1)^2 x^2+2 (s+1) y & x^3-x^4-3 x^2 y-y^2-s^4 x^4-4 s^3 x^4-6 s^2 x^4-3 s^2 x^2 y-4 s x^4-6 s x^2 y \\
\end{array}
\right)
$$
}}
and $\widehat{h}=\adj(\widehat{d})$.
Notice that making $s=0$ in $\widehat{d}, \widehat{h}$ and $\widehat{F}$ we get $d, h$ and $F$ of \exref{exC5MF}.
\end{example}

\medskip
In the following example, inspired in \cite[Example 1]{EU97}, we show that there exists
matrix factorizations $(d,h)$ of a convergent series $F\in \mathbb C\{x,y\}$ such that $\coker(d)$ is not a $\mathbb C$-algebra.
We will face this problem in the next section.

\medskip
\begin{example}
\label{no-alg}
Let $(Y,0)$ be the monomial plane curve singularity  with parametrization $x=t^3, y=t^4$, so
$\mathcal O_{(Y,0)}= \mathbb C\{x,y\}/(x^4-y^3)\cong\mathbb C\{t^3,t^4\}$.
We consider the $\mathcal O_{(Y,0)}$-module
$B=\mathcal O_{(Y,0)} + t \mathcal O_{(Y,0)}\subset \mathbb C\{t\}$.
Notice that $B$ is generated as $\mathbb C$-vector space by
$1,t, t^3,t^4,t^5,\dots$ and that $B$ is not a ring.
We have that the pair $(d,h)$ is a matrix factorization of $F=x^4-y^3$ where
$$
d = \left(
  \begin{array}{cc}
    y &-x^3 \\
    -x & y^2 \\
  \end{array}
\right)
$$
and $h=\adj(d)$, inducing a minimal resolution of $B$ as $\mathbb C\{x,y\}$-module:
$$
0\longrightarrow
\mathbb C\{x,y\}^2\stackrel{d}{\longrightarrow}
\mathbb C\{x,y\}^2\stackrel{(1,t)}{\longrightarrow}
B
\longrightarrow 0.
$$
\end{example}

\medskip
\section{Matrix decompositions associated to a plane curve}

In this section we study the family of curve singularities $(X,0)\subset \ori{n}$, and their associated matrix factorizations, sharing a given generic plane projection $(Y,0)\subset \ori{2}$.

Following the notations of the previous section, we know that there are matrix factorizations such that $\coker(d)$ is not a $\mathbb C\{x,y\}$-algebra and, in particular, $\coker(d)$ is not the ring of regular functions of a  curve singularity, \exref{no-alg}.
The first step of this section is to characterize the matrix factorizations coming from a  curve singularity.

First, we give an upper bound of the first Betti number of   matrix factorizations defined by a curve singularity.
This result gives a partial answer to the question appearing in page $4$ of the first paragraph of \cite{EP16}.

\medskip
\begin{lemma}
\label{bound-b}
Let $(X,0)$ be a curve singularity of $\ori{n}$ and let $(Y,0)$ be a generic plane projection of $(X,0)$.
Then
$$
\beta_{\mathcal O_{(Y,0)}}(\mathcal O_{(X,0)})\le e_0(X,0)-1.
$$
\end{lemma}
\begin{proof}
We may assume that  $(X,0)$ has, after a good election of the uniformation parameter,
a standard parametrization
$$
(X,0): \left\{
 \begin{array}{ll}
x_1=t^e&\\
x_2=f_2(t)&\\
\vdots & \\
x_n=f_n(t)
 \end{array}
\right.
$$
with $e< v_t(f_2(t))\le \dots \le v_t(f_n(t))$ and $e=e_0(X,0)$.
The generic plane projection of $(X,0)$ from $L=(c_1,\dots, c_{2n})\in W_{(X,0)}$ has a parametrization
$$
\left\{
  \begin{array}{ll}
x=c_1 t^e+ \sum_{j=2}^{n} c_j f_j(t)& \\ \\
y=c_{n+1} t^e+ \sum_{j=n+2}^{2n} c_j f_{j-n}(t)\\
 \end{array}
\right.
$$
see proof of \propref{generic-proj}.

Assume that $c_1=c_{n+1}=0$, then the $(n-2)$
-plane $H$ defined by $L$
has equations
$c_2 x_2+\cdots +c_n x_n=0$,
$c_{n+2} x_2+\cdots +c_{2n}x_n=0$.
Notice that the tangent line of $(X,0)$ is contained in $H$ contradicting that $L\in W_{(X,0)}$,
\propref{cone5}.
Hence $c_1\neq 0$ or $c_{n+1}\neq 0$, so $v_t(x)=e$ or $v_t(y)=e$.
This implies that $x$, or $y$, is a superficial element of degree one of $\mathcal O_{(X,0)}$ and then
$\mathcal O_{(X,0)}/a \mathcal O_{(X,0)}=e$, for $a=x$ or $a=y$.
Since $x ,y$ belong to a minimal system of generators of $\max_{(X,0)}$ and form a system of generators of $\max_{(Y,0)}$, we have
$$
\beta_{\mathcal O_{(Y,0)}}(\mathcal O_{(X,0)})=
\dim_{\mathbb C}\left(\frac{\mathcal O_{(X,0)}}{(x,y)\mathcal O_{(X,0)}}\right)<
\dim_{\mathbb C}\left(\frac{\mathcal O_{(X,0)}}{a\mathcal O_{(X,0)}}\right)=
e_0(X,0)
$$
for
$a=x$ or $a=y$.
\end{proof}

\medskip
Next we characterize matrix factorizations $(d,h)$ of a power  series $F\in\mathbb C\{x,y\}$ for which  $\coker{(d)}$ is a $\mathbb C\{x,y\}$-algebra, see \exref{no-alg}.
We follow the ideas appearing in the introduction of \cite{EU97}.

\medskip
\begin{theorem}
\label{char-alg}
Let $(Y,0)\subset (\mathbb C^2,0)$ be an irreducible plane curve singularity defined by an equation $F\in\mathbb C\{x,y\}$.
Let $(d,h)$ be a matrix factorization of $F$ inducing a free resolution of $\coker{(d)}$ as
$\mathbb C\{x,y\}$-modules
$$
0\longrightarrow
\mathbb C\{x,y\}^b\stackrel{d}{\longrightarrow}
\mathbb C\{x,y\}^b\longrightarrow
\coker(d)
\longrightarrow 0.
$$
Then, $\coker{(d)}$ is a finitely generated faithful $\mathcal O_{(Y,0)}$-module and
the following conditions are equivalent:
\begin{enumerate}
\item[(i)] there is a finitely $\mathcal O_{(Y,0)}$-algebra $B$, isomorphic to   $\coker(d)$ as
$\mathcal O_{(Y,0)}$-module, and $\mathcal O_{(Y,0)}$-algebra extensions
$$
\mathcal O_{(Y,0)}\subset  B \subset \overline{\mathcal O_{(Y,0)}}\cong \mathbb C\{t\}
$$

\item[(ii)] there exists $\alpha\in \mathcal O_{(Y,0)}\setminus \{0\}$ and
$e\in \coker{(d)}\setminus \{0\}$ such that
$$
\alpha \coker{(d)}\subset e \mathcal O_{(Y,0)}\subset \coker{(d)}
$$
and
the natural morphism
$\coker{(d)}\rightarrow \ext^1_{\mathcal O_{(Y,0)}}(\coker{(d)}/e \mathcal O_{(Y,0)},\coker{(d)})$ induced by the exact sequence
$$
 0\longrightarrow
e \mathcal O_{(Y,0)}\longrightarrow
\coker{(d)}         \longrightarrow
\coker{(d)}/e \mathcal O_{(Y,0)}\longrightarrow
0
$$
is zero.
\end{enumerate}
If these equivalent conditions hold then  $B=e^{-1}\coker(d)$ is a finitely generated $\mathcal O_{(Y,0)}$-algebra local domain of dimension one.
Moreover, there is a curve singularity $(X,0)\subset (\mathbb C^{b+2},0)$ such that
$\mathcal O_{(X,0)}= B$.
\end{theorem}
\begin{proof}
We know that the matrix factorization induces a free periodic resolution of $\mathcal O_{(Y,0)}$-modules
$$
 \cdots \longrightarrow
\mathcal O_{(Y,0)}^b \stackrel{h}{\longrightarrow}
\mathcal O_{(Y,0)}^b \stackrel{d}{\longrightarrow}
\mathcal O_{(Y,0)}^b \stackrel{\pi}{\longrightarrow}
\coker{(d)}
\longrightarrow 0,
$$
\cite[Theorem 2.1.1]{EP16}.
Then $\coker{(d)}$ is a finitely generated $\mathcal O_{(Y,0)}$-module.

Next, we prove that $\coker(d)$ is a faithful $\mathcal O_{(Y,0)}$-module.
Let $a\in \mathcal O_{(Y,0)}$ be a non-zero element.
Assume that there is $g\in \coker{(d)}\setminus \{0\}$ such that $ag=0$.
Since we have a commutative diagram
\begin{center}
\begin{tikzcd}
\cdots  {\mathcal O}_{(Y,0)}^b\arrow[ r, "h"] \arrow[d,"a"]&    {\mathcal O}_{(Y,0)}^b\arrow[ r, "d"] \arrow[d,"a"] & {\mathcal O}_{(Y,0)}^b\arrow[ r, "\pi"] \arrow[d,"a"] &\coker{(d)}\arrow[d,"a"]\arrow[r]&0\\
\cdots  {\mathcal O}_{(Y,0)}^b\arrow[ r, "h"] &  {\mathcal O}_{(Y,0)}^b\arrow[ r, "d"] & {\mathcal O}_{(Y,0)}^b\arrow[ r, "\pi"]&\coker{(d)}\arrow[r]&0
\end{tikzcd}
\end{center}
we get that there is $g'\in {\mathcal O}_{(Y,0)}^b$ such that $\pi(g')=g$.
Hence $\pi(a g')=a \pi(g')=a g=0$ and there is $g''\in {\mathcal O}_{(Y,0)}^b$
with $d(g'')=a g'$.
Moreover
$$
0=hd(g'')=h(ag')=a h(g'),
$$
since ${\mathcal O}_{(Y,0)}^b$ is a domain and $a\neq 0$ we get $h(g')=0$.
Then there is $w\in {\mathcal O}_{(Y,0)}^b$ with $g'=d(w)$, so
$$
g=\pi(g')=\pi(d(w))=0.
$$
Hence  $\coker{(d)}$ is a faithful $\mathcal O_{(Y,0)}$-module.

\medskip
\noindent
Let us assume $(i)$.
Then, identifying  $\coker(d)$ with $B$, we have  $\mathcal O_{(Y,0)}$-algebra extensions
$$
\mathcal O_{(Y,0)}\subset  \coker(d) \subset \overline{\mathcal O_{(Y,0)}}\cong \mathbb C\{t\}.
$$
We take $e=1$ and $\alpha$ a generator of the conductor of the extension
 $\mathcal O_{(Y,0)} \subset \mathbb C\{t\}$.

Applying the functor $\Hom_{\mathcal O_{(Y,0)}}(\cdot, \coker(d))$ to the  exact sequence
$$
 0\longrightarrow
\mathcal O_{(Y,0)}\stackrel{i}{\longrightarrow}
\coker{(d)}       \stackrel{\pi}{\longrightarrow}
\coker{(d)}/\mathcal O_{(Y,0)}\longrightarrow
0
$$
we get the exact sequence
\begin{multline*}
 0\longrightarrow
\Hom_{\mathcal O_{(Y,0)}}(\coker{(d)}/\mathcal O_{(Y,0)}, \coker(d))
  \stackrel{\pi^*}{\longrightarrow}
\Hom_{\mathcal O_{(Y,0)}}(\coker{(d)}, \coker(d))\stackrel{i^*}{\longrightarrow}\\
  \stackrel{i^*}{\longrightarrow}
\Hom_{\mathcal O_{(Y,0)}}(\mathcal O_{(Y,0)}, \coker(d))\cong  \coker(d)
  \stackrel{\phi}{\longrightarrow}
\ext^1_{\mathcal O_{(Y,0)}}(\coker{(d)}/\mathcal O_{(Y,0)},\coker{(d)}).
\end{multline*}
Notice that $\phi=0$  if and only if any $\mathcal O_{(Y,0)}-$morphism
$f:\mathcal O_{(Y,0)}\longrightarrow \coker(d)$ can be lifted to a $\mathcal O_{(Y,0)}-$morphism
$\tilde{f}:\coker(d) \longrightarrow \coker(d)$.
We can define
$\tilde{f}(q)=\alpha^{-1} f(\alpha q)$,
for all $q\in \coker(d)$.
Since $\alpha q\in \mathcal O_{(Y,0)}$
$$
\tilde{f}(q)=\alpha^{-1} f(\alpha q)=\alpha^{-1}\alpha qf(1)= qf(1)\in \coker(d).
$$
Being $f$ a $\mathcal O_{(Y,0)}$-linear map, $\tilde{f}$ is $\mathcal O_{(Y,0)}$-linear
as well.

\medskip
\noindent
Assume now $(ii)$.
We have an exact sequence of $ \mathcal O_{(Y,0)}$-modules:
$$
 0\longrightarrow
e \mathcal O_{(Y,0)}\stackrel{i}{\longrightarrow}
\coker{(d)}       \stackrel{\pi}{\longrightarrow}
\coker{(d)}/\mathcal O_{(Y,0)}\longrightarrow
0,
$$
inducing the exact sequence
\begin{multline*}
 0\longrightarrow
\Hom_{\mathcal O_{(Y,0)}}(\coker{(d)}/e \mathcal O_{(Y,0)}, \coker(d))
  \stackrel{\pi^*}{\longrightarrow}
\Hom_{\mathcal O_{(Y,0)}}(\coker{(d)}, \coker(d))\stackrel{i^*}{\longrightarrow}\\
  \stackrel{i^*}{\longrightarrow}
\Hom_{\mathcal O_{(Y,0)}}(e \mathcal O_{(Y,0)}, \coker(d))\cong \coker(d)
  \stackrel{\phi}{\longrightarrow}
\ext^1_{\mathcal O_{(Y,0)}}(\coker{(d)}/e\mathcal O_{(Y,0)},\coker{(d)}).
\end{multline*}
Since  $\phi=0$,  any $\mathcal O_{(Y,0)}-$morphism
$f:e \mathcal O_{(Y,0)}\longrightarrow \coker(d)$ can be lifted to a $\mathcal O_{(Y,0)}$-morphism
$\tilde{f}:\coker(d) \longrightarrow \coker(d)$.
On the other hand, since $\alpha \coker{(d)}\subset e \mathcal O_{(Y,0)}$
and $\alpha \in \mathcal O_{(Y,0)}\setminus \{0\}$
we get that
$\Hom_{\mathcal O_{(Y,0)}}(\coker{(d)}/e \mathcal O_{(Y,0)}, \coker(d))=0$.
Hence we have the isomorphism
$$
\Hom_{\mathcal O_{(Y,0)}}(\coker{(d)}, \coker(d))\stackrel{i^*}{\cong}
\Hom_{\mathcal O_{(Y,0)}}(e \mathcal O_{(Y,0)}, \coker(d)).
$$

\noindent
\begin{lemma}
\label{lem-est}
 $\coker(d)$ admits a structure of $\mathcal O_{(Y,0)}$-algebra with $e$ as identity element.
\end{lemma}
\begin{proof}
Recall that, since $\coker(d)$ is a faithful $\mathcal O_{(Y,0)}$-module, if  $e\omega =e\omega'$
then $\omega=\omega'$.
For all $q\in \coker{(d)}$ we consider the $\mathcal O_{(Y,0)}$-morphism
$f_q:e \mathcal O_{(Y,0)}\longrightarrow \coker{(d)}$ defined by
$f_q(e w)= w  q$, for all $w\in \mathcal O_{(Y,0)}$.
Hence there exists a $\mathcal O_{(Y,0)}$-morphism
$\tilde{f}_q:\coker(d) \longrightarrow \coker(d)$ extending $f_q$.
The  structure of $\coker(d)$ as $\mathcal O_{(Y,0)}$-algebra is defined by:
for all $q_1,q_2\in \coker(d)$,
$$
q_1 q_2=\tilde{f}_{q_1}(q_2)=\tilde{f}_{q_2}(q_1).
$$
If we write $\alpha q_1=e a$, $a\in \mathcal O_{(Y,0)}$ 
then
$$
\alpha^2 \tilde{f}_{q_2}(q_1)=\alpha \tilde{f}_{q_2}(\alpha q_1)=
\alpha f_{q_2}(ea)=\alpha a q_2.
$$
By symmetry, $\alpha q_2=e b$, $b\in \mathcal O_{(Y,0)}$,
$$
\alpha^2 \tilde{f}_{q_1}(q_2)=\alpha \tilde{f}_{q_1}(\alpha q_2)=
\alpha f_{q_1}(eb)=\alpha b q_1.
$$
Hence
$$
\alpha^2 \tilde{f}_{q_2}(q_1)= e a b =\alpha^2 \tilde{f}_{q_1}(q_2)
$$
Since $\coker(d)$ is a faithful $\mathcal O_{(Y,0)}$-module and $\alpha \neq 0$ we get
$
\tilde{f}_{q_2}(q_1) =\tilde{f}_{q_1}(q_2).
$
We let to the reader the proof that the previous definition makes $\coker(d)$
a $\mathcal O_{(Y,0)}$-algebra with $e$ as identity element.
\end{proof}

\medskip
Notice that from the inclusion $\alpha \coker(d)\subset e \mathcal O_{(Y,0)}$
we deduce that
$$
\coker(d)\subset e \alpha^{-1} \mathcal O_{(Y,0)} \subset e K(\mathcal O_{(Y,0)})
$$
where $K(\mathcal O_{(Y,0)})$ is the ring of fractions of $\mathcal O_{(Y,0)}$.
Then we define
$$
B=e^{-1}\coker(d)\subset K(\mathcal O_{(Y,0)}).
$$
Since $e$ is the unit element of $\coker(d)$, $B$ is an $\mathcal O_{(Y,0)}$-algebra.
On the other hand, since $B$ is a finitely generated $\mathcal O_{(Y,0)}$-module we get that
$B\subset \overline{\mathcal O_{(Y,0)}}=\mathbb C\{t\}$.

Assume now that any of the two above equivalent conditions hold.
We proved that $B$ is a finitely generated sub-$\mathcal O_{(Y,0)}$-algebra of $K(\mathcal O_{(Y,0)})$ so $B$ is a domain and the extension $\mathcal O_{(Y,0)}\subset B$ is finite.
Then $B$ is a Cohen-Macaulay local ring of dimension one.
Hence there exists a curve singularity $(X,0)\subset (\mathbb C^{b+2},0)$ such that
$\mathcal O_{(X,0)} = B$.
\end{proof}

\medskip
\begin{remark}
Assume  $(ii)$ of the last result holds.
Since $\coker(d)\cong B$ is  a one-dimensional local domain  then
$\length_{\mathcal O_{(Y,0)}}(\coker{(d)}/e \mathcal O_{(Y,0)})< +\infty$
for any $e\in \coker(d)\setminus \{0\}$, and then
there exists
$\alpha\in \mathcal O_{(Y,0)}\setminus \{0\}$  such that
$$
\alpha \coker{(d)}\subset e \mathcal O_{(Y,0)}\subset \coker{(d)}.
$$
On the other hand, if $\length_{\mathcal O_{(Y,0)}}(\coker{(d)}/e \mathcal O_{(Y,0)})< +\infty$
for an element $e\in \coker(d)\setminus \{0\}$ then
there exists $\alpha\in \mathcal O_{(Y,0)}\setminus \{0\}$ satisfying the above two inclusions.
The second condition of $(ii)$ on the $\ext^1$ is easy computable.
\end{remark}

\medskip
\begin{remark}
With the notations of \exref{no-alg}.
We know that the $\mathcal O_{(Y,0)}$-module $B=\mathcal O_{(Y,0)} + t \mathcal O_{(Y,0)}\subset \mathbb C\{t\}$ is not a ring.
We can check this fact by using the last result.
The second part of $(ii)$ of the last theorem is equivalent to: any morphism
$f:\mathcal O_{(Y,0)}\longrightarrow B$ can we lifted to a morphism
$\tilde{f}:B\longrightarrow B$.
If $f$ is the natural inclusion then there is not a such $\tilde{f}$.
In fact, we have $t x -y=0$ in $B$, so
$0=\tilde{f}(tx-y)=\tilde{f}(t) x-f(y)=\tilde{f}(t)t^4-t^3$.
Hence $\tilde{f}(t)t-1=0$, this is not possible in $\mathbb C\{t\}$.
\end{remark}

\medskip
\begin{remark}
\label{alg-est}
Assume that we have the following $\mathcal O_{(Y,0)}$-module extensions, where $(Y,0)$ is a plane curve singularity,
$$
\mathcal O_{(Y,0)}\stackrel{i}{\hookrightarrow} B \hookrightarrow \mathbb C\{t\}=\overline{\mathcal O_{(Y,0)}}.
$$
Assume that $B$ is an $\mathcal O_{(Y,0)}$-algebra with $i$ its the structural morphism.
This structure is unique because  is the induced by
$B\subset B[c^{-1}]=\mathcal O_{(Y,0)}[c^{-1}]$ where $c$ is a generator
of the conductor of the ring extension
$\mathcal O_{(Y,0)}\subset \mathbb C\{t\}$.
Notice that this structure can be recovered as follows, see \lemref{lem-est}.
Given $q\in B$ we consider the $\mathcal O_{(Y,0)}$-linear map
$f_q:\mathcal O_{(Y,0)}\longrightarrow B$ defined by
$f_q(b)=b q$, for all $b\in \mathcal O_{(Y,0)}$.
This map has a trivial  lifting $\tilde{f_q}:B\longrightarrow B$ defined by
$\tilde{f_q}(b)=bq$, for all $b\in B$.
Then the product in $B$ is
$$
q_1q_2=\tilde{f_{q_1}}(q_2)=\tilde{f_{q_2}}(q_1).
$$
Notice that the lifting $\tilde{f_q}$ is unique since the extension $\mathcal O_{(Y,0)}\subset \mathbb C\{t\}$ has a non-trivial conductor.
\end{remark}

\medskip
In the following result we refine \thmref{char-alg} characterizing the matrix factorization defined by generic plane projections.

\medskip
\begin{proposition}
\label{MatRep-proj}
Let $(Y,0)\subset (\mathbb C^2,0)$ be an irreducible plane curve singularity defined by an equation $F\in\mathbb C\{x,y\}$.
Let $(d,h)$ be a matrix factorization of $F$ such that $\coker{(d)}$ is a $\mathcal O_{(Y,0)}$-algebra  satisfying the equivalent conditions of \thmref{char-alg}.
Let $(X,0)\subset (\mathbb C^{2+b},0)$ be a curve singularity such that $\mathcal O_{(X,0)}\cong \coker(d)$ as $\mathbb C$-algebras with $b=\beta_{\mathcal O_{(Y,0)}}(\coker(d))$, \thmref{char-alg}.
The following conditions are equivalent:
\begin{enumerate}
\item[(i)] $(Y,0)$ is the plane generic projection of $(X,0)$,
\item[(ii)]
the cosets
$\overline{x},\overline{y}\in \mathcal O_{(X,0)}$ are $\mathbb C$-linear independent
modulo $\max_{(X,0)}^2$ and $\mu(Y)=\overline{\mu}(X)$.
\end{enumerate}
\end{proposition}
\begin{proof}
We write $n=2+b$.
Assume that $(Y,0)$ is the plane generic projection of $(X,0)$.
From \thmref{generic-proj} we get $\mu(Y)=\overline{\mu}(X)$.
We know that the morphism
$$\pi^*: \mathcal O_{(Y,0)}\longrightarrow \mathcal O_{(X,0)}=\mathbb C\{x_1,\dots,x_n\}/I_X,$$ with $I_X\subset (x_1,\dots,x_n)^2$,
satisfies
$\pi^*(\overline{x})=z_1\overline{x_1}+\dots+z_n \overline{x_n}$ and   $\pi^*(\overline{y})=z_{n+1}\overline{x_1}+\dots+z_{2 n} \overline{x_n}$, for some
$z_1,\dots,z_{2n}\in\mathbb C$, and the linear variety
$z_1x_1+\dots+z_n x_n=z_{n+1}x_1+\dots+z_{2 n} x_n=0$ is of dimension $n-2$.
Hence the cosets $\overline{x},\overline{y}\in \mathcal O_{(X,0)}$ are $\mathbb C$-linear independent modulo $\max_{(X,0)}^2$ and we get $(ii)$.

\noindent
Assume now $(ii)$.
Notice that $\overline{x},\overline{y}\in \mathcal O_{(X,0)}$ are $\mathbb C$-linear independent modulo $\max_{(X,0)}^2$ and $\mu(Y)=\overline{\mu}(X)$.
Then, after a change of variables, we may assume that $\max_{(X,0)}/\max_{(X,0)}^2$ is generated by the cosets of $x,y,x_3,\dots,x_n$, i.e.
$\mathcal O_{(X,0)}\cong \mathbb C\{x,y,x_3,\dots,x_n\}/I_X$ with $I_X\subset (x,y,x_3,\dots,x_n)^2 \mathbb C\{x,y,x_3,\dots,x_n\}$.
Hence $(Y,0)$ is a linear projection of $(X,0)$ and since $Y$ is $\mu(Y)=\overline{\mu}(X)$
we get that $(Y,0)$ is the plane generic projection of $(X,0)$, \cite[Proposition IV.2 b]{BG80}.
\end{proof}

\medskip
Next we compare two curve singularities sharing the same generic plane projection.

\medskip
\begin{proposition}
\label{isos}
Let $(Y,0)$ be a plane curve singularity.
Let $B_1$, $B_2$ be two $\mathcal O_{(Y,0)}$-algebras
$$
\mathcal O_{(Y,0)}\stackrel{i}{\hookrightarrow} B_i \hookrightarrow \mathbb C\{t\}=\overline{\mathcal O_{(Y,0)}}
$$
$i=1,2$.
A $\mathcal O_{(Y,0)}$-linear map $\phi:B_1\longrightarrow B_2$ is an isomorphism as $\mathcal O_{(Y,0)}$-modules if and only if $\phi$ is an isomorphism of  $\mathcal O_{(Y,0)}$-algebras.
\end{proposition}
\begin{proof}
If $\phi$ is an isomorphism $\mathcal O_{(Y,0)}$-algebras then it is an isomorphism of $\mathcal O_{(Y,0)}$-modules.

Assume now that $\phi$ is an isomorphism of $\mathcal O_{(Y,0)}$-modules.
I will use the notations of \remref{alg-est}.
First we prove that for all $q_1,q_2\in B_1$ it holds
$\phi(q_1q_2)=\phi(q_1)\phi(q_2)$.
Since
$$
\phi(q_1q_2)=\phi( \widetilde{f_{q_1}}(q_2)) \text{\quad    and  \quad }
\phi(q_1)\phi(q_2)=\widetilde{f_{\phi(q_1)}}(\phi(q_2))
$$
so we have to prove that
$$
\phi( \widetilde{f_{q_1}}(q_2))=\widetilde{f_{\phi(q_1)}}(\phi(q_2)).
$$
We can take as lifting of $f_{\phi(q_1)}$
$$
\widetilde{f_{\phi(q_1)}}=\phi \circ \widetilde{f_{q_1}}  \circ \phi^{-1}.
$$
Then
$$
\widetilde{f_{\phi(q_1)}}(\phi(q_2))=
\phi \circ \widetilde{f_{q_1}}  \circ \phi^{-1}\phi(q_2) =
\phi \circ \widetilde{f_{q_1}}  (q_2) =
\phi( \widetilde{f_{q_1}}(q_2)).
$$
Notice that $\phi\mid_{\mathcal O_{(Y,0)}}=Id_{\mathcal O_{(Y,0)}}$.
\end{proof}

\medskip
\begin{definition}
\label{MatRep}
Given an irreducible series $F\in \mathbb C\{x,y\}$ and an  integer $b\ge 2$ we denote by
$\mr_{F,b}$ the set of pairs of dimension $b$ square matrices $(d,h)$ with entries in $(x,y)\mathbb C\{x,y\}$ such that
$$
h d = d h= F \; \mathrm{Id}_{\mathbb C\{x,y\}^b},
$$
and $\coker(d)$ satisfies the equivalent properties of  \thmref{char-alg} and \propref{MatRep-proj}, i.e.
$\coker{(d)}$ is isomorphic to $\mathcal O_{(X,0)}$ as $\mathcal O_{(Y,0)}$-modules,where $(X,0)$ is a curve singularity with $F$ as the equation defining a generic plane projection $(Y,0)$ of $(X,0)$.
\end{definition}

\medskip
We denote by   $Gl_b(\mathbb C\{x,y\})$ the group of invertible $b \times b$ matrices with entries in $\mathbb C\{x,y\}$;
we consider the following action in $\mr_{F,b}$
$$
\begin{array}{ccc}
  \mr_{F,b} \times  Gl_b(\mathbb C\{x,y\})^2& \stackrel{\circ}{\longrightarrow} &   \mr_{F,b}\\ \\
((d,h),(\phi,\psi)) & \mapsto & (\phi,\psi)\circ (d,h)=(\phi d\psi^{-1},\psi h \phi^{-1})
\end{array}
$$

\medskip
\begin{definition}
Given a plane branch singularity $(Y,0)\subset \ori{2}$ we denote by
$\mathcal G_{(Y,0),b}^n$, $n\ge3$, the set of curve singularities $(X,0)\subset \ori{n}$
such that $(Y,0)$ is a generic plane projection of $(X,0)$ and
$b=\beta_{\mathcal O_{(Y,0)}}(\mathcal O_{(X,0)})$.
For each such a curve singularity $(X,0)$ we choose a projection $\pi_L:(X,0)\longrightarrow (Y,0)$,
$L\in W_{(X,0)}$, equivalently an immersion
$\pi_L^*:  \mathcal O_{(Y,0)}\longrightarrow \mathcal O_{(X,0)}$.
\end{definition}

\medskip
We consider in $\mathcal G_{(Y,0),b}^n$, $n\ge3$ the following binary equivalence relation: given $(X_i,0)\in \mathcal G_{(Y,0),b}^n$, $i=1,2$,
we write $(X_1,0)  \thicksim (X_2,0)$ if and only if
$\mathcal O_{(X_1,0)}\cong \mathcal O_{(X_2,0)}$ as $\mathcal O_{(Y,0)}$-algebras, see \propref{isos}.

\medskip
\begin{remark}
Recall that two curve singularities $(X_i,0)\subset (\mathbb C^n,0)$, $i=1,2$, share
the same generic plane projection if and only if $(X_1,0)$ and $(X_2,0)$ have isomorphic saturations, \cite[Theorem VI.0.2]{BGG80}.
\end{remark}

\medskip
\begin{theorem}
\label{bijection}
There is a bijective  map
$$
\begin{array}{ccc}
 \mathcal G_{(Y,0),b}^n / \thicksim & \stackrel{\epsilon}{\longrightarrow} &   \mr_{F,b}/Gl_b(\mathbb C\{x,y\})^2\\ \\
(X,0) & \mapsto & \overline{(d,h)}
\end{array}
$$
where $(d,h)$ is a matrix factorization of $F$, where $F$ is an equation defining $(Y,0)$.
\end{theorem}
\begin{proof}
First we prove that the map $\epsilon$ is defined.
We already proved in \thmref{matfact} that for all $(X,0)\in \mathcal G_{(Y,0),b}^n$ there exists a matrix factorization $(d,h)$ of $F$ inducing a minimal free resolution
$$
0\longrightarrow
\mathbb C\{x,y\}^b\stackrel{d}{\longrightarrow}
\mathbb C\{x,y\}^b\longrightarrow
\mathcal O_{(X,0)}
\longrightarrow 0
$$
Given a second matrix factorization $(d',h')$ of $F$ we have a second minimal resolution
$$
0\longrightarrow
\mathbb C\{x,y\}^b\stackrel{d'}{\longrightarrow}
\mathbb C\{x,y\}^b\longrightarrow
\mathcal O_{(X,0)}
\longrightarrow 0
$$
Both minimal resolution are isomorphic so there are $\phi, \psi\in Gl_b(\mathbb C\{x,y\})^2$
such that $\phi d= d' \psi$.
Then $\phi d \psi^{-1}= d'$ and
$\psi h \phi^{-1}=h'$, so  $\overline{(d,h)}=\overline{(d',h')}$.

The map $\epsilon$ is exhaustive.
Let $\overline{(d,h)}$ be an element of $\mr_{F,b}/Gl_b(\mathbb C\{x,y\})^2$.
That means that $\coker{(d)}\cong \mathcal O_{(X,0)}$ as $\mathcal O_{(Y,0)}$-modules, where
$(X,0)$ is a curve singularity with $(Y,0)$ as a generic plane projection.
Hence $\epsilon(X,0)= \overline{(d,h)}$.

The map $\epsilon$ is injective.
Let $(X_1,0)$, $(X_2,0)$ be curve singularities of $ \mathcal G_{(Y,0),b}^n$  such that $\epsilon(X_1,0)=\epsilon(X_2,0)$.
This implies that
$\mathcal O_{(X_1,0)} \cong \mathcal O_{(X_2,0)}$ as $\mathcal O_{(Y,0)}$-modules.
From \propref{isos} that isomorphism is of $\mathcal O_{(Y,0)}$-algebras, so $\epsilon$ is injective.
\end{proof}



\bibliographystyle{amsplain}

\medskip
\noindent
Joan Elias\\
Departament de Matemàtiques i Informàtica\\
Universitat de Barcelona\\
Gran Via 585, 08007 Barcelona, Spain\\
e-mail: {\tt elias@ub.edu}
\end{document}